\documentclass[11pt]{amsart}
\usepackage{amsmath}
\usepackage{amsfonts}
\usepackage{amssymb}
\usepackage{amsthm}
\usepackage{url}
\usepackage[all]{xy}
\usepackage{dsfont}
\usepackage{graphicx}
\usepackage{caption}
\usepackage{subcaption}
\usepackage{comment} 
\usepackage{stmaryrd}
\usepackage{hyperref}
\usepackage{todonotes}
\usepackage{color}
\usepackage{enumerate,tikz-cd}
\usepackage[backend=biber,style=ieee-alphabetic,maxbibnames=10,hyperref=true,doi=false,url=false,isbn=false,sorting=nty]{biblatex}
\usepackage{comment}
\allowdisplaybreaks

\numberwithin{equation}{section}

\newtheorem{proposition}{Proposition}[section]

\newtheorem{lemma}[proposition]{Lemma}
\newtheorem{theorem}[proposition]{Theorem}
\newtheorem{corollary}[proposition]{Corollary}

\theoremstyle{definition}
\newtheorem{remark}[proposition]{Remark}
\newtheorem{definition}[proposition]{Definition}
\newtheorem{example}[proposition]{Example}

\DeclareMathOperator{\End}{End}

\DeclareMathOperator{\tr}{tr}
\DeclareMathOperator{\GL}{GL}

\DeclareMathOperator{\SL}{SL}
\DeclareMathOperator{\Proj}{Proj}
\DeclareMathOperator{\Spec}{Spec}

\newcommand{\R}{\mathbb{R}}
\newcommand{\C}{\mathbb{C}}

\newcommand{\Z}{\mathbb{Z}}
\newcommand{\Q}{\mathbb{Q}}

\renewcommand{\P}{\mathbb{P}}

\renewcommand{\O}{\mathcal{O}}
\renewcommand{\H}{\mathcal{H}}
\newcommand{\E}{\mathcal{E}}
\newcommand{\F}{\mathcal{F}}
\renewcommand{\L}{\mathcal{L}}
\newcommand{\X}{\mathcal{X}}

\newcommand\vol{\mathrm{vol}}
\newcommand\triv{{\mathrm{triv}}}
\newcommand\NA{{\mathrm{NA}}}

\title[A lower bound on the Calabi functional for a degeneration]{A lower bound on the Calabi functional for a degeneration of polarized varieties}

\author[Gabriel Frey]{Gabriel Frey}

\address{Mathematics Institute,  University of Warwick, Coventry CV4 7AL, UK} \email{gabriel.frey@warwick.ac.uk}
\begin{document}

\begin{abstract}

We prove a lower bound on the Calabi functional for degenerations of polarized varieties, involving the difference of CM degrees between generically isomorphic families.
This may be viewed as a discretely valued version of Donaldson's lower bound for models, in the sense of non-Archimedean geometry.
In particular, this generalizes a result of Donaldson, who considered a single polarized variety.
As a main tool, we develop the theory of GIT height, introduced by Wang, and apply it to the family GIT problem of the Chow variety.
Using the GIT height, we also give a numerical proof of separatedness of GIT quotients for general and special linear actions, strengthening prior work of Wang--Xu.

\end{abstract}

\maketitle

\section{Introduction}

The study of families and degenerations of projective varieties plays a crucial role in understanding moduli and the geometry of polarized varieties.
In our case, we consider complex polarized varieties, which are projective varieties over $\C$ endowed with a choice of ample line bundle.

One important question is the existence of special K\"ahler metrics and its relation with algebro-geometric stability conditions.
The famous Yau--Tian--Donaldson conjecture \cite{Yau92,Tia97,Don02} predicts that, given a smooth polarized variety $(X,L)$, there exists a K\"ahler metric $\omega$ in the first Chern class $c_1(L)$ with constant scalar curvature $S(\omega)$, if and only if $(X,L)$ is K-polystable.
The latter is measured using test configurations, which may be viewed as a $\C^*$-equivariant degenerations of $(X,L)$.
This conjecture has recently been resolved by Boucksom--Jonsson \cite{BJ25} and Darvas--Zhang \cite{DZ25}, extending a classical result by Chen--Donaldson--Sun \cite{CDS15} in the Fano case, using non-Archimedean geometry by enlarging the space of test configurations to a completion $\E^{1,\NA}_L$.

A classical result due to Donaldson \cite{Don05} provides the following lower bound on the Calabi functional associated to a polarized manifold:

\begin{theorem}\label{thm : Donaldsons lower bound on the Calabi functional}
    Given a polarized manifold $(X,L)$, we have
    \begin{equation*}
        \inf_{\omega \in \H_L} \big\| S(\omega) - \widehat{S} \big\|_{L^2(X,\omega)}
        \geq - \sup_{(\X,\L)} \frac{\mathrm{DF}(\X,\L)}{\|(\X,\L) \|_2},
    \end{equation*}
    where $\H_L$ denotes the space of K\"ahler metrics in $c_1(L)$ and the supremum is taken over the space of non-trivial test configurations $(\X,\L)$ of $(X,L)$.
\end{theorem}

Here, $\widehat{S}$ is the average scalar curvature, while $\mathrm{DF}(\X,\L)$ and $\|(\X,\L)\|_2$ denote the Donaldson--Futaki invariant and $L^2$-norm of $(\X,\L)$, respectively \cite{Don02} and \cite[Section 3]{Sze06}.
Note that if $\H_L$ contains a cscK metric, Donaldson's lower bound implies that the Donaldson--Futaki invariant of any test configuration is non-negative, meaning that $(X,L)$ is K-semistable.

Theorem \eqref{thm : Donaldsons lower bound on the Calabi functional} has been inspired by the moment-weight inequality in geometric invariant theory (GIT) \cite{MFK94}, which leads to similar inequalities in other geometric settings, such as for vector bundles due to Atiyah--Bott \cite{AB83}, where equality is obtained by taking the supremum over slope-decreasing flags, or in the K\"ahler setting due to Dervan \cite{Der18}.
Donaldson conjectured that equality is also true for manifolds by taking the supremum over the space of test configurations.
As Theorem \eqref{thm : Donaldsons lower bound on the Calabi functional} is tautological when $(X,L)$ is K-semistable, the conjecture may be viewed as a ``K-unstable version'' of the Yau--Tian--Donaldson conjecture.
Xia \cite{Xia21} proved that equality is obtained by increasing the space of test configurations to the space of $E^2$-geodesic maximal rays.

The Donaldson--Futaki invariant of a test configuration is the difference of CM degrees between the restriction and the trivial family over the entire disk $\Delta$ \cite[Section 3.2]{Wan12}.
In particular, Theorem \eqref{thm : Donaldsons lower bound on the Calabi functional} may be written in terms of differences of CM degrees with respect to the trivial filling.
In fact, given two fillings $(\X,\L)$ and $(\X',\L')$ of a family $(X,L) \to \Delta^*$ over the punctured disk, Blum--Xu \cite{BX19} and Hattori \cite{Hat24} constructed a ``good'' filtration $\mathcal{F}_{\L,\L'}$ of the section ring of the central fiber $(\X_0,\L_0)$, such that
\begin{equation}\label{eq : DF of Hattoris filtration}
    \mathrm{DF}(\mathcal{F}_{\L,\L'})
    = \mathrm{CM}(\X',\L')
    - \mathrm{CM}(\X,\L),
\end{equation}
where the Donaldson--Futaki invariant is defined as in \cite{Hat24}.

The goal of the present paper is to prove a general version of Donaldson's lower bound on the Calabi functional, which applies to arbitrary degenerations of smooth polarized varieties.
As a tool, we provide a GIT version of Equation \eqref{eq : DF of Hattoris filtration}, using the GIT weight and height, the latter being a notion introduced by Wang \cite{Wan12}.

\subsection{Main results}

Let $R$ be a complete discrete valuation ring (DVR) with fraction field $K$ and residue field $\C$.
For example, the completion $R = \widehat{\O}_{C,0}$ of the local ring of an algebraic curve $C$ at a closed point $0$ is a DVR.

Let $(X,L)$ be a polarized $K$-variety and let $(\X,\L)$ be an $R$-model of $(X,L)$.
We may think of $(X,L)$ as a degeneration of a polarized variety over $\Spec K$, and of $(\X,\L)$ as a flat filling of $(X,L)$ over $\Spec R$.

The first main result of the present paper is as follows:
\begin{theorem}\label{thm : main theorem}
    If the central fiber $\X_0$ is smooth, we have
    \begin{equation*}
        \inf_{\omega \in \H_{\L_0}}
        \big\| S(\omega) - \widehat{S} \big\|_{ L^2(\X_0,\omega)}
        \geq
        - \sup_{(\X',\L')} \frac{\mathrm{CM}(\X',\L') - \mathrm{CM}(\X,\L)}{\hat{d}_2 \big( (\X',\L') , (\X,\L)\big)},
    \end{equation*}
    where the supremum is taken over the space of models $(\X',\L')$ of $(X,L)$ not isomorphic to $(\X,\L)$.
\end{theorem}
Since $\Spec R$ is not proper, the degree of the CM line bundle is not well-defined; however, one may identify the difference of their degrees as a rational number.
The denominator is a distance between models, defined using non-Archimedean norms, which generalizes the norm of a test configuration.
Similarly, the enumerator may also be viewed as a non-Archimedean Mabuchi-type energy associated to the model metrics (see Section \ref{subsec : nA interpretation of results}).

Theorem \ref{thm : main theorem} may be viewed as a proof of one direction of the ``K-unstable version'' of the Yau--Tian--Donaldson Conjecture, for a degeneration of polarized varieties.
In particular, we conjecture that equality is obtained provided one includes finite-energy non-Archimedean metrics, analogously to the conjecture of Donaldson and the work of Xia.

Note that one cannot apply Theorem \ref{thm : Donaldsons lower bound on the Calabi functional} and Equation \eqref{eq : DF of Hattoris filtration} in order to prove Theorem \ref{thm : main theorem}, as the filtration might not be finitely-generated.
In fact, we use a quantization strategy in the discretely valued case (see Section \ref{subsec : strategy of the proof}).

Similarly to Donaldson's lower bound on the Calabi functional, Theorem \ref{thm : main theorem} implies the following stability result for cscK manifolds, strengthening work of Hattori \cite{Hat24} and in \cite{Fre25}.
\begin{corollary}
    If $(\X_0,\L_0)$ is smooth and admits a cscK metric, then
    \begin{equation*}
        \mathrm{CM}(\X,\L)
        \leq \mathrm{CM}(\X',\L'),
    \end{equation*}
    for any model $(\X',\L')$ of $(X,L)$.
\end{corollary}

The second main result is a GIT version of Equation \eqref{eq : DF of Hattoris filtration} and a version of \textit{(strict) GIT height minimization} in the discretely valued case.
Let $G = \GL_n$ or $\SL_n$ acting linearly on projective space $(\P^N,\O(1))$.
\begin{theorem}\label{thm : intro GIT height minimization}
    Let $\varphi$ an $\psi$ be $R$-points of $\P^N$ and assume that there exists a $K$-point of $G$ with $g.\varphi_K = \psi_K$ in $\P^N_K$.
    There is a one-parameter subgroup $\rho : \C^* \to G$, whose GIT weight at $\varphi(0)$ is equal to the difference of heights
    \begin{equation*}
        \nu(\rho,\varphi(0))
        = h(\psi) - h(\varphi).
    \end{equation*}
    Moreover, if $\varphi(0)$ is polystable, then
    \begin{equation*}
        h(\varphi)
        \leq h(\psi),
    \end{equation*}
    with equality if and only if $\varphi(0)$ is contained in $\overline{G.\psi(0)}$.
\end{theorem}

The difference of GIT heights is defined using the degree of line bundles over the spectrum of a discrete valuation ring, generalizing the definition due to Wang \cite{Wan12}.
Theorem \ref{thm : intro GIT height minimization} improves work of Wang \cite{Wan12} and Wang--Xu \cite{WX12} by allowing families over non-proper bases and strict minimization.
The construction of the one-parameter subgroup $\rho$ is explicit and in the spirit of the filtration constructed by Blum--Xu \cite{BX19}.

Furthermore, Theorem \ref{thm : intro GIT height minimization} provides a numerical proof of separatedness for GIT quotients of projective schemes by $\GL$ or $\SL$-actions.
In particular, it may be viewed as a GIT version of a separatedness and uniqueness of K-polystable degenerations of Fano varieties due to Blum--Xu \cite{BX19}.

\subsection{Strategy of the proof}\label{subsec : strategy of the proof}

We briefly outline the strategy for proving Theorem \ref{thm : main theorem}, which revolves around establishing asymptotic upper and lower bounds on the Chow moment map
\begin{equation*}
    \mu_\mathrm{Chow} : \mathrm{Chow}(\P(V)) \to \mathfrak{su}(V)^\vee,
\end{equation*}
given a unitary vector space $V$, associated to the $\mathrm{SU}(V)$-action on the Chow variety $\mathrm{Chow}(\P(V))$.
Note that $\mathfrak{su}(V)$ admits a $2$-norm $\|A\|_2 = -\mathrm{tr}(A^2)$.

Given a flat relatively polarized family $(\X,\L) \to \Spec R$, one obtains a sequence of embeddings $\X_k \subset \P(V_k) \times \Spec R$ of $\X$ for $V_k = H^0(\X_0,k\L_0)$ with $k$ sufficiently large.
The central fiber $\X_{k,0}$ is an $n$-dimensional subvariety of $\P(V_k)$ and gives rise to a closed point $[\X_{k,0}]$ in $\mathrm{Chow}(\P(V_k))$.

The first step is identical to a result due to Donaldson \cite[Proposition 1]{Don05}, using the Bergman kernel, which is the asymptotic upper bound
\begin{equation*}
    \| \mu_\mathrm{Chow}([\X_{k,0}]) \|_2
    \leq \sqrt{a_0} \big\|S(\omega) - \widehat{S} \big\|_{L^2(\X_0, \omega)} k^{\frac{n}{2}-1}
    + O(k^{\frac{n}{2} -2})
\end{equation*}
as $k \to +\infty$, where $a_0 = \int_{\X_0} \omega^n/n!$ comes from the $L^2$-normalization.

The second step requires new GIT weight and height techniques, with the goal of achieving the asymptotic lower bound
\begin{equation*}
    \|\mu_\mathrm{Chow}([\X_{k,0}])\|_2
    \geq - \sqrt{a_0}\frac{\mathrm{CM}(\X',\L') - \mathrm{CM}(\X,\L)}{\hat{d}_2 \big( (\X',\L') , (\X,\L) \big)} k^{\frac{n}{2}-1}
    + o(k^{\frac{n}{2}-1})
\end{equation*}
as $k \to +\infty$, given any model $(\X',\L')$ not isomorphic to $(\X,\L)$.
Note that combining both asymptotic bounds proves Theorem \ref{thm : main theorem}.

The technique involved in the second step is a finite-dimensional argument, strengthening work due to Donaldson \cite[Proposition 2]{Don05}, involving the \textit{difference of Chow heights}, generalizing the notion of Chow weights of test configurations, and improving theory by Wang \cite{Wan12} to the discretely valued setting.
This is achieved by using Theorem \ref{thm : intro GIT height minimization}, to construct a one-parameter subgroup $\rho_k : \C^* \to \GL(V_k)$ with the property that its GIT weight at the point $[\X_{k,0}]$ is equal to the difference of Chow heights.
This construction is the heart of the paper and works in general for $\GL$ and $\SL$-actions on projective space.

In fact, similar to \cite{Don05}, we prove a slightly more general version of Theorem \ref{thm : main theorem} (see Theorem \ref{thm : refined main theorem in Section 4}) involving the $L^q$-norms and the $d_p$-distance, where $1 < p,q < +\infty$ are H\"older conjugate (meaning that $\frac{1}{p} + \frac{1}{q} = 1$).

\subsection{Structure of the paper}

This paper is divided into three sections.
In the second section, we recall the necessary prerequisites on degrees of line bundles over a discrete valuation ring, models, differences of CM degrees, GIT weight and height and the Bergman kernel expansion.

In the third section, we establish the second step outlined the strategy of the proof of the main theorem.
Given a linear $\GL_n$-action on $(\P^N,\O(1))$, we provide a general construction of a one-parameter subgroup out of two $R$-points in $\P^N(R)$ lying in the same $\mathrm{GL}_n(K)$-orbit and fixing an identification, such that its GIT weight is the difference of GIT heights. 
We also provide general applications, such as GIT height minimization (Theorem \ref{thm : intro GIT height minimization}), which serves as a numerical proof for separatedness of GIT quotients.

In the fourth and final section, we complete the proof of the main theorem by combining the results of the previous two sections.
We also give a non-Archimedean interpretation of the result in terms of the Mabuchi energy in the discretely-valued case.

\subsection{Acknowledgments}

I would like to thank my PhD supervisor Ruadha\'i Dervan for countless helpful conversations and support during the present work.
I also would like to thank Mattias Jonsson and R\'emi Reboulet for helpful comments and the University of Michigan for hospitality during my visit to Ann Arbor.
I received funding from a PhD studentship associated to Dervan's Royal Society University Research Fellowship (URF\textbackslash R1\textbackslash 201041).

\section{Preliminaries}\label{sec : preliminaries}

Throughout the section, we fix a complete discrete valuation ring $R$ with fraction field $K$, residue field $\C$, and choose a valuation $\mathsf{v} : K^* \to \R$.
A \textit{uniformizer} for $(R,\mathsf{v})$ is a non-zero element $t$ in $R$ that generates the valuation group $\mathsf{v}(K^*)$.
A uniformizer always exists, but does not have to be unique.

Furthermore, recall that the spectrum $\Spec R$ consists of a generic open point $\Spec K$ and a special closed point $0$, corresponding to the residue field.

\subsection{Degrees of line bundles over DVRs}

We start by defining the degree of a line bundle over a discrete valuation ring. Since $\Spec R$ is not proper, one requires to fix a trivialization of the line bundle over the generic fiber to define the degree.
This construction of the degree has been studied by Donaldson \cite{Don12} for line bundles over the open disk $\Delta$ in $\C$, and the analogue for for discrete valuation rings is closely related to relative K-theory \cite[Chapter 18.2]{Ful13}.

Let $\lambda$ be a line bundle on $\Spec R$ and choose an isomorphism of line bundles $\lambda|_{\Spec K} \simeq \O_{\Spec K}$ over the generic point. This choice gives rise to an isomorphism of vector spaces
\begin{equation*}
    \Phi : H^0(\Spec K, \lambda) \to K,
\end{equation*}
where the domain contains $H^0(\Spec R, \lambda)$ as an $R$-submodule.
\begin{definition}\label{def : degree of line bundle over dvr}
    The \textit{degree} of $\lambda$ with respect to $\Phi$ is the valuation
    \begin{equation*}
        \deg_\Phi \lambda := \mathsf{v}(\gamma),
    \end{equation*}
    where $\gamma$ is any non-zero element in $K$ such that $\gamma \cdot R = \Phi(H^0(\Spec R, \lambda))$.
\end{definition}

Alternatively, given a uniformizer $t$, the degree of $\lambda$ with respect to $\Phi$ is the unique maximal integer $d$ such that, given any trivializing section $s$ in $H^0(\Spec R, \lambda)$, the element $t^{-d}\Phi(s)$ is contained in $R$.

\begin{example}\label{exp : deg l = deg_C lbar}
    In the case when $R = \widehat{\O}_{C,0}$ is the completion of the local ring of a proper curve $C$ at a point $0$, one always has a canonical choice of valuation by taking the degree at $0$.
    One may compactify $\lambda$ to a line bundle $\bar{\lambda}$ on $C$ by gluing $\lambda$ on $\Spec R$ to the trivial line bundle on $C \setminus 0$ along the identification on the intersection $\Spec K$. The degree of $\lambda$ then is
    \begin{equation*}
        \deg_\Phi(\lambda) = \deg_C (\overline{\lambda} ).
    \end{equation*}
\end{example}

The definition of the degree satisfies some basic properties, which we implicitly make use of:
given line bundles $\lambda$ and $\lambda'$ on $\Spec R$ with trivialization $\Phi$ and $\Phi'$ away from the central fiber, we have natural induced trivializations
\begin{align*}
    \Phi^{-1} &: H^0(\Spec K,-\lambda)
    \to K,
    \\
    \Phi \otimes \Phi' &: H^0(\Spec K,\lambda + \lambda') \to K,
\end{align*}
where $-\lambda$ denotes the dual of $\lambda$.
By definition of the degree, we obtain
\begin{align*}
    \deg_{\Phi^{-1}}(-\lambda)
    &= - \deg_\Phi \lambda,
    \\
    \deg_{\Phi \otimes \Phi'} (\lambda + \lambda')
    &= \deg_\Phi(\lambda) + \deg_{\Phi'}(\lambda').
\end{align*}

Definition \ref{def : degree of line bundle over dvr} naturally extends to $\Q$-line bundles $\lambda$ on $\Spec R$: let $m$ be any positive integer such that $m\lambda$ is a line bundle, using additive notation for tensor products of line bundles. Given an identification $\lambda|_{\Spec K} \simeq \O_{\Spec K}$, one obtains an isomorphism $\Phi : H^0(\Spec K,m\lambda) \to K$ for $m \geq 1$ sufficiently divisible.
We then define
\begin{equation*}
    \deg_\Phi \lambda = m^{-1} \deg_\Phi (m\lambda),
\end{equation*}
noting that this is independent of the choice of $m$ by additivity.

In Sections \ref{subsec : bergman kernel} and \ref{subsec : applic to Chow pts} we will study models embedded into projective space, where this point of view is more convenient.

\subsection{Models and the difference of CM degrees}\label{subsec : models and cm degree}

We introduce the main objects of interest in the present paper, namely models over discrete valuation rings and the Chow--Mumford line bundle.
Let $(X,L)$ be an $n$-dimensional normal polarized $K$-variety.
This means that $X$ is a normal integral scheme over $\Spec K$, the structure morphism $X \to \Spec K$ is flat, of finite type and of relative dimension $n$, and that $L$ is a relatively ample $\Q$-line bundle on $X$.

An \textit{$R$-model} of $(X,L)$ is a flat normal relatively $\Q$-polarized integral scheme $(\mathcal{X},\mathcal{L})$ of finite type and projective over $\Spec R$, together with an isomorphism $(\mathcal{X}_K,\mathcal{L}_K) \simeq (X,L)$ over $\Spec K$.
Here, $(\mathcal{X}_K,\mathcal{L}_K)$ denotes the base change of $(\mathcal{X},\mathcal{L})$ over the generic point.

Given a model $\pi : (\X,\L) \to \Spec R$, the Knudsen--Mumford expansion \cite[Theorem 4]{KM76} (see also \cite{FR06}) constructs $\Q$-line bundles $\lambda_0$ and $\lambda_1$ on $\Spec R$, fitting into a functorial isomorphism
\begin{equation}\label{eq : Knudsen-Mumford expansion}
    \det \pi_* (k\mathcal{L})
    \simeq \frac{k^{n+1}}{(n+1)!}\lambda_0
    - \frac{k^n}{2n!}\lambda_1
    + O(k^{n-1}).
\end{equation}
For expansions of line bundles, $O(k^{n-1})$ denotes a linear combination consisting of (functorial) $\Q$-line bundles and factors of $k^m$ for $m \leq n-1$.

We define the \textit{Chow--Mumford (CM) line bundle} of $(\X,\L)$ to be the $\Q$-line bundle
\begin{equation}\label{eq : def CM line bundle}
        \lambda_\mathrm{CM}(\mathcal{X},\mathcal{L})
        = \frac{1}{2} \Big(
        \lambda_1 
        - \frac{n\mu }{n+1} \lambda_0 
        \Big),
    \end{equation}
where $V = n!a_0$ and $\mu= -(2a_1)/(na_0)$ denote the volume and slope of the central fiber, given by the coefficients of the asymptotic expansion of the Hilbert polynomial $h^0(\mathcal{X}_0,k\mathcal{L}_0) = a_0 k^n + a_1 k^{n-1} + O(k^{n-2})$ as $k \to +\infty$.

The functoriality of the Knudsen--Mumford expansion assures that, given another model $(\X',\L')$, the isomorphisms $(\X_K,\L_K) \simeq (X,L) \simeq (\X'_K,\L'_K)$ induce an explicit isomorphism
\begin{equation*}
    \lambda_\mathrm{CM}(\X,\L)|_{\Spec K}
    \simeq \lambda_\mathrm{CM}(\X',\L')|_{\Spec K}.
\end{equation*}
The \textit{difference of CM degrees} between the two models $(\mathcal{X},\mathcal{L})$ and $(\mathcal{X}', \mathcal{L}')$ of $(X,L)$ is defined to be the degree
\begin{equation*}
    \mathrm{CM}(\mathcal{X}, \mathcal{L})
    - \mathrm{CM}(\mathcal{X}',\mathcal{L}')
    := \deg \big( \lambda_\mathrm{CM}(\X,\L) - \lambda_\mathrm{CM}(\X',\L') \big),
\end{equation*}
with respect to the identification induced by the isomorphisms between the CM line bundles away from the central fiber.
When comparing degrees, we write $\mathrm{CM}(\mathcal{X}, \mathcal{L}) \geq \mathrm{CM}(\mathcal{X}',\mathcal{L}')$, whenever $\mathrm{CM}(\mathcal{X}, \mathcal{L}) - \mathrm{CM}(\mathcal{X}',\mathcal{L}') \geq 0$, and similarly for other inequalities.

\subsubsection{Models embedded into projective space}

For all integers $k$, due to flatness of the structure morphism $\pi : \X \to \Spec R$, one may choose a trivialization of the vector bundle $\pi_*(k\L)^\vee \cong V_k \times \Spec R$, where $V_k = H^0(\X_0,k\L_0)$.
Since we assume $\L$ to be relatively ample, for $k$ sufficiently large, this trivialization gives rise to an embedding
\begin{equation*}
    \iota_k : \X \hookrightarrow \P(V_k) \times \Spec R,
\end{equation*}
with the property that $k\L = \iota^*_k \O(1)$, the pull-back of the $\O(1)$-line bundle on $\P(V_k) \times \Spec R$.
In particular, the relatively polarized $R$-scheme $(\X,\L)$ is isomorphic to $(\X_k,k^{-1}\O_{\X_k}(1))$, where $\X_k := \iota_k(\X)$ is a normal integral subscheme of $\P(V_k) \times \Spec R$.

The embedding of models into projective space allows us to classify models in terms of points in the general linear group $\GL(V_k)$.
For this, we say that two $K$-points $g$ and $g'$ in $\mathrm{GL}(V_k)$ are \textit{equivalent} if their product $g^{-1} g'$ extends to an $R$-point.
\begin{lemma}\label{lem : models and K-points of GL}
    There is a one-to-one correspondence between models $(\X',\L')$ of $(X,L)$ up to isomorphism, for which $k\L'$ is relatively very ample, and equivalence classes of $K$-points of $\GL(V_k)$.
\end{lemma}
This was proven by Donaldson \cite[Section 2.1]{Don15} in the case when the family $(\X,\L)$ is a trivial product, where this is a statement about the theory of arcs (see also \cite[Proposition 2.6]{DR24}).
The same proof by Dervan--Reboulet applies here.

In light of Lemma \ref{lem : models and K-points of GL}, when studying models, one may equivalently use flat normal integral subschemes $\X$ of $\P(V) \times \Spec R$ for a finite-dimensional complex vector space $V$, where the relative ample line bundle is the $\O_\X(1)$ line bundle coming from the embedding.
Furthermore, when studying two models of a given family, one may equivalently use two flat normal integral subschemes $\X$ and $\X'$ of $\P(V) \times \Spec R$, together with a $K$-point $g$ of $\GL(V)$ such that the corresponding automorphism $V \otimes_\C K \to V \otimes_\C K$ restricts to an isomorphism $(\X_K,\O_{\X_K}(1)) \to (\X'_K,\O_{\X'_K}(1))$.

\subsubsection{Distances between models}\label{subsubsec : dist between models}

In this section, we define the distance between two models in terms of non-Archimedean norms. The general relevant theory may be found in many works such as \cite{BHJ17,BE21,Reb21}.
We will give more comments on the non-Archimedean aspect of the difference of CM degrees and the main result in Section \ref{subsec : nA interpretation of results}.

Given a model $(\X,\L)$ of $(X,L)$, one may define a norm $\| \cdot \|_{\L,k}$ on $H^0(X,kL)$, with respect to the non-Archimedean absolute value $|\cdot| = e^{-\mathsf{v}}$ on $K$, by declaring
\begin{equation*}
    \|s\|_{\L,k}
    = \sup \left\{ |\alpha| \;\middle|\; \alpha \in K \text{ and } s \in \alpha \cdot H^0(\X,k\L) \right\},
\end{equation*}
for all sections $s$ in $H^0(X,kL)$.
Here, we viewed $H^0(\X,k\L)$ canonically as an $R$-submodule of $H^0(X,kL)$, using the identification $(\X_K,\L_K) \simeq (X,L)$ coming from the model.
This norm satisfies the ultrametric inequality turning it into a non-Archimedean norm.
The family of norms $\{ \| \cdot \|_{\L,k} \}$ is submultiplicative, bounded and graded (see \cite[Section 1.3]{Reb21}), and we denote it by $\phi_\L$, by slight abuse of notation, compared to the literature.

Given two families $\phi = \{ \| \cdot \|_{\phi,k} \}$ and $\phi' = \{ \| \cdot \|_{\phi',k} \}$ of bounded graded norms on $\bigoplus_{k \geq 0} H^0(X,kL)$, the limit
\begin{equation*}
    \hat{d}^\NA_p ( \phi, \phi)
    = \lim_{k \to +\infty} \frac{\hat{d}_p(\|\cdot\|_{\phi,k}, \|\cdot\|_{\phi',k}) }{k},
\end{equation*}
exists, where we define the \textit{normalized $d_p$-distance} between two norms to be
\begin{equation*} 
    \hat{d}_p(\|\cdot\|_{\phi,k}, \|\cdot\|_{\phi',k})^p 
    = \frac{1}{n_k} \sum_{i = 1}^{n_k} | \lambda_i - \bar{\lambda} |^p.
\end{equation*}
Here, we wrote $n_k = \dim H^0(X,kL)$ and $\lambda_i = \log \| e_i \|_{\phi',k} - \log \|e_i\|_{\phi,k}$ and $\bar{\lambda} = n_k^{-1}\sum_{i = 1}^{n_k} \lambda_i$, given any basis $(e_i)$ of $H^0(X,kL)$ that simultaneously diagonalizes $\| \cdot \|_{\phi,k}$ and $\| \cdot \|_{\phi',k}$ \cite[Sections 1.2 and 1.3]{Reb21}.

\begin{definition}\label{def : distance between models}
    The \textit{normalized $d_p$-distance} between two models $(\X',\L')$ and $(\X,\L)$ of $(X,L)$ is the non-Archimedean distance
    \begin{equation*}
        \hat{d}_p \big( (\X',\L') , (\X,\L) \big)
        = \hat{d}_p^\mathrm{NA} (\phi_{\L'}, \phi_\L).
    \end{equation*}
\end{definition}
The normalized $d_p$-distance is a metric in the usual sense on the set of (normal) models of $(X,L)$, much as the $d_p^\NA$-distance is a metric on the space $\H^\NA_L$ of families of bounded graded norms, modulo asymptotic equivalence \cite[Section 2.4]{Reb21}.

\subsection{GIT weight and height}\label{subsec : GIT height}

We introduce the basic concepts of geometric invariant theory (GIT) for a reductive algebraic group acting on a polarized variety.
Most of the theory of this is well-established (see, for example, \cite{MFK94, GRS21,Hos15}), and this section is more intended to clarify the notation and sign conventions used in the present paper.
We also introduce the notion of GIT heights, introduced by Wang \cite{Wan12}, which will play a central role in Section \ref{sec : fin dim arguments}.

Let $G$ be a complex reductive algebraic group acting on a complex polarized variety $(X,L)$.
This means that $G$ acts algebraically on $X$ and the action is endowed with a linearization on $L$.

\subsubsection{GIT weight}

The \textit{GIT weight} of a point $x$ in $X$ with respect to a one-parameter subgroup $\rho : \C^* \to G$ is defined to be the weight $\nu(\rho,x)$ of the $\C^*$-action on the fiber $L_{x_0}$, given by $\rho$, at the specialization $x_0 = \lim_{t \to 0} \rho(t).x$.
We call the point $x$
\begin{itemize}
    \item \textit{semistable}, if $\nu(\rho,x) \geq 0$ for any one-parameter subgroup $\rho$ of $G$,

    \item \textit{polystable}, it is semistable and $\nu(\rho,x)=0$ implies that $\rho(\C^*) \subset G_x$,

    \item \textit{stable}, if $\nu(\rho,x) > 0$ for any non-trivial one-parameter subgroup $\rho$,
\end{itemize}
where $G_x$ denotes the stabilizer of $x$ in $G$.

In the case when $X$ is smooth, the GIT weight may also be interpreted in terms of moment maps.
Fix a maximal compact subgroup $K$ of $G$ whose Lie algebra $\mathfrak{k}$ complexifies to the complex Lie algebra $\mathfrak{g}$ of $G$.
Assume that there is a K\"ahler metric $\omega$ in $c_1(L)$, such that the $K$-action on $(X,\omega)$ is Hamiltonian and fix a $K$-invariant hermitian metric $h$ on $L$ with curvature $\frac{i}{2\pi}\omega$.
Define a map $\mu : X \to \mathfrak{k}^\vee$, called the \textit{(normalized) moment map} of the GIT problem, given by
\begin{equation}\label{eq : moment map for GIT}
    \mu(x)(A) =  \frac{\langle \xi, \xi_A \rangle_h}{i\langle \xi,\xi \rangle_h},
\end{equation}
for all points $x$ and Lie algebra elements $A$, where $\xi$ is any non-zero element in the fiber $L_x$, and $\xi_A = \frac{d}{dt}|_{t = 0} \exp(tA).\xi$.

Given a one-parameter subgroup $\rho : \C^* \to G$ satisfying $\rho(S^1) \subset K$, the infinitesimal generator $A = \frac{d}{dt}|_{t = 0} \rho(e^t)$ is contained in $i\mathfrak{k}$.
In that case,
\begin{equation}\label{eq : GIT weight = moment map}
    \nu(\rho,x)
    = -\mu(x_0)(iA),
\end{equation}
where $x_0 = \lim_{z \to 0} \rho(z).x$.
Using Equation \ref{eq : GIT weight = moment map}, one may prove the \textit{moment-weight inequality}
\begin{equation}\label{eq : moment-weight ineq}
    \inf_{g \in G} \|\mu(g.x)\|
    \geq -\sup_{\rho} \frac{\nu(\rho,x)}{\|\rho\|},
\end{equation}
given any $K$-invariant norm $\|\cdot\|$ on $\mathfrak{k}$ \cite[Chapter 8]{GRS21}, where the supremum is taken over the space of non-trivial one-parameter subgroups of $G$ and writing $\| \rho \| = \| A \|$.
This may be viewed as a lower bound of the moment map functional in the GIT setting.

\subsubsection{GIT height}

We continue with introducing the notion of GIT height, which is a geometric object associated to a family of points in $X$.
Let $\varphi, \psi : \Spec R \to X$ be $R$-points of $X$ and assume that there exists a $K$-point $g$ in $G$ such that $g \circ \varphi_K = \psi_K$, where $\varphi_K$ and $\psi_K$ denote the restrictions to $\Spec K$.
This gives rise to an identification
\begin{equation*}
    H^0(\Spec K, \varphi^*L) \to H^0(\Spec K, \psi^*L),
    \quad \sigma \mapsto g \circ \sigma,
\end{equation*}
inducing an isomorphism $\Phi_g : H^0(\Spec K, \varphi^*L - \psi^*L) \to K$.

\begin{definition}
    The \textit{GIT height} of $\varphi$ is the pull-back line bundle $\varphi^*L$ on $\Spec R$.
    The \textit{difference of GIT heights} between $\varphi$ and $\psi$, with respect to the identification $g$, is defined to be the degree
    \begin{equation*}
        h(\varphi) - h(\psi)
        := \deg_{\Phi_g} ( \varphi^*L - \psi^*L ).
    \end{equation*}
    We write $h(\varphi) \geq h(\psi)$ when $h(\varphi) - h(\psi) \geq 0$, and similarly for other inequalities.
\end{definition}

In light of Example \ref{exp : deg l = deg_C lbar} and additivity of the degree, this definition matches the definition of Wang \cite[Definition 3]{Wan12} for the difference of heights over a proper base.

\begin{example}
    In the case $R = \C [\![ t ]\!]$ and $K = \C (\!( t )\!)$, given a one-parameter subgroup $\rho : \C^* \to G$ and a point $x$ in $X$, we obtain a $R$-point $\rho_K . x$ in $X(K) = X(R)$ by separatedness and properness of $X$.
    The GIT weight of $\rho$ at $x$ then is the difference of GIT heights
    \begin{equation*}
        \nu(\rho,x)
        = h(\rho_K.x) - h(x_R),
    \end{equation*}
    between $\rho_K.x$ and the constant $R$-point $x_R$ with value $x$, with respect to the identification $g = \rho^{-1}_K$.
\end{example}

\subsection{The Bergman kernel estimate}\label{subsec : bergman kernel}

The first step of proving Theorem \ref{thm : main theorem} consists of a  differential geometric asymptotic expansion relating the Calabi functional associated to the central fiber $(\X_0,\L_0)$ of a model $(\X,\L)$ and the moment map norm of the sequence of embeddings of $\X_0$ into projective space $\P(H^0(\X_0,k\L_0)^\vee)$.
This is the same expansion as in \cite[Proposition 1]{Don05}, which is proven using the Bergman kernel, both of which we will recite in this section.

Let $V$ be a finite-dimensional complex vector space and choose a hermitian inner product $\langle \cdot , \cdot \rangle$ on $V$.
Write $\mathrm{SU}(V)$ for the Lie group of special unitary endomorphisms of $V$ with respect to this inner product.
Its Lie algebra $\mathfrak{su}(V)$ consists of trace-free skew-hermitian endomorphisms of $V$ and admits a $q$-norm, given by $\|A\|^q_q = \tr(A^q)$ for any $1 < q < +\infty$.
Furthermore, let $\mathrm{Chow}(\P(V))$ be the Chow variety parameterizing effective cycles in $\P(V)$, which admits an ample line bundle $\Lambda_\mathrm{Chow}$ \cite[Section I.3]{Kol13} and a K\"ahler metric $\omega_\mathrm{Chow}$ in $c_1(\Lambda_\mathrm{Chow})$.

Define $\mu_\mathrm{FS} : \P(V) \to \mathfrak{su}(V)^\vee$ and $\mu_\mathrm{Chow} : \mathrm{Chow}(\P(V)) \to \mathfrak{su}(V)^\vee$ by
\begin{align*}
    \mu_\mathrm{FS}([v])(A)
    &= \frac{\langle v,Av \rangle}{i\langle v, v \rangle},
    \\
    \mu_\mathrm{Chow}([Z])(A)
    &= \int_Z \langle \mu_\mathrm{FS},A \rangle \, \frac{\omega^n_\mathrm{FS}}{n!},
\end{align*}
for all $v$ in $V$ and $A$ and $\mathfrak{su}(V)$, and $n$-dimensional effective cycles $Z$ in $\P(V)$, where $\omega_\mathrm{FS}$ is the Fubini--Study metric given by the choice of inner product.
In fact, these maps are moment maps for the corresponding $\mathrm{SU}(V)$-actions on $(\P(V),\omega_\mathrm{FS})$ and $(\mathrm{Chow}(\P(V)),\omega_\mathrm{Chow})$ \cite[Section 2.1]{Don01}.

Let $(\X,\L)$ be a model of a normal $\Q$-polarized $K$-variety $(X,L)$ and repeat the construction from Section \ref{subsec : models and cm degree}, writing $\X_k = \iota_k(\X)$ for an embedding $\iota_k : \X \hookrightarrow \P(V_k) \times \Spec R$, given by a trivialization $\pi_*(k\L)^\vee \cong V_k \times \Spec R$ for $V_k = H^0(\X_0,k\L_0)$.
Assume that $\X_0$ is smooth, and choose a K\"ahler metric $\omega$ in $c_1(\L_0)$ and a hermitian metric $h$ on $\L_0$ with curvature $\frac{i}{2\pi}\omega$. This induces $L^2$-inner products $\langle \cdot,\cdot \rangle$ on $V_k$, so that $\mathrm{SU}(V_k)$ is defined.
Note that the $p$-norm on $\mathfrak{su}(V_k)$ naturally induces a $p$-norm on its dual.

\begin{theorem}[{\cite[Proposition 1]{Don05}}]\label{thm : donaldsons bergman kernel expansion}
    As $k \to +\infty$, we have
    \begin{equation*} 
        \left\|\mu_\mathrm{Chow}([\X_{k,0}])\right\|_q
        \leq \sqrt[q]{a_0} \big\|S(\omega) - \widehat{S}\big\|_{L^q(\X_0, \omega)} k^{\frac{n}{q}-1}
        + O(k^{\frac{n}{q} -2}),
    \end{equation*}
    where we wrote $h^0(\X_0,k\L_0) = a_0 k^n + O(k^{n-1})$.
\end{theorem}

The factor of $\sqrt[q]{a_0}$ comes from the normalization for the $L^q$-norm, which is taken with respect to the measure $\omega^n/\int_{\X_0}\omega^n$.

Theorem \ref{thm : donaldsons bergman kernel expansion} is proven using an asymptotic expansion of the Bergman kernel (see \cite[Chapter 7]{Sze14}).
The \textit{($k$\textsuperscript{th}-)Bergman kernel} of $\X_0$ with respect to the hermitian metric $h$ is the smooth function $B_{k,h} : \X_0 \to \R$, given by
\begin{equation*}
    B_{k,h}(x)
    = \sum_{i = 0}^{N_k} |s_i(x)|^2_h,
\end{equation*}
for all $x$ in $\X_0$, where $(s_i)$ is any choice of unitary basis of $V_k$.
This function has the property that
\begin{align*}
    \iota^*_k \omega_\mathrm{FS}
    &= 2\pi k\omega + i \partial \overline{\partial} \log B_{k,h},
    \\
    B_{k,h}(x)
    &= 1 + \frac{S(\omega)(x)}{2 \pi k} + O(k^{-2}),
\end{align*}
for all $x$ in $\X_0$, a result due to \cite{Tia90,Rua98,Lu00,Cat99}.
This also shows that $a_0 = \int_{\X_0} \omega^n / n!$ using the asymptotic Riemann--Roch formula.

In Section \ref{subsec : applic to Chow pts}, we wish to use an alternative description of the Chow moment map in terms of Chow points in the Grassmannian $\mathrm{Gr}(V,r)$ of $r$-dimensional linear subspaces in $V$.
Let
\begin{equation*}
    H = H^0 \big( \mathrm{Gr}(V,N-n),\O(d) \big),
\end{equation*}
for integers $n \geq 0$ and $d \geq 1$, which admits a natural hermitian inner product induced by the inner product on $V$, as well as a natural $\GL(V)$-action, which restricts to a hermitian $\mathrm{U}(V)$-action with respect to the inner products.

Given an effective cycle $Z$ in $\P(V)$ of dimension $n$ and degree $d$, we define the \textit{Chow point} of $Z$ to be the unique point $R_Z$ in $\P(H)$, associated to the divisor
\begin{equation*}
    Z(R_Z)
    = \big\{ [\Lambda] \in \mathrm{Gr} (V,N-n)
    \;\big|\; \P(\Lambda) \cap Z \neq \varnothing \big\},
\end{equation*}
(see \cite[Chapter 3, Section B]{GKZ08}).
The construction of Chow points may be improved to an embedding of varieties
\begin{equation*}
    \mathrm{Chow}_{n,d}(\P(V)) \to \P(H),
\end{equation*}
sending an effective cycle of dimension $n$ and degree $d$ to its Chow point \cite[Corollary 3.23.3]{Kol13}.
This means that, given a $\C$-scheme $S$ and an integral subscheme $\mathcal{Z}$ of $\P(V) \times S$, flat over $S$ and such that its fibers are of dimension $n$ and degree $d$, we obtain a morphism $R_\mathcal{Z} : S \to \P(H)$, called the \textit{Chow ($S$-)point of $\mathcal{Z}$}.
This morphism has the property that it sends a schematic point $\xi$ in $S$ to the Chow point $R_{\mathcal{Z}_\xi}$ in $\P(H \otimes_\C \kappa(\xi))$ of the subvariety $\mathcal{Z}_\xi$ in $\P(V \otimes_\C \kappa(\xi))$, where $\kappa(\xi)$ is the residue field of $\xi$.

A basic property of the morphism between the Chow variety and $\P(H)$ is that it is $\GL(V)$-equivariant and respects the K\"ahler metrics.
This means that the moment map $\mu : \P(H) \to \mathfrak{su}(V)^\vee$ from Equation \eqref{eq : moment map for GIT}, given by
\begin{equation}\label{eq : moment map mu_H}
    \mu([\sigma])(A)
    = \frac{\langle \sigma, \sigma_A \rangle}{i\langle \sigma, \sigma \rangle},
\end{equation}
for all $\sigma$ in $H$ and $A$ in $\mathfrak{su}(V)$, where $\sigma_A = \frac{d}{dt}|_{t = 0} \exp(tA).\sigma$, satisfies
\begin{equation}\label{eq : moment map between chow var and chow point}
    \mu_\mathrm{Chow}([Z])
    = \frac{d}{n!}\cdot \mu ( R_Z ),
\end{equation}
for all $[Z]$ in $\mathrm{Chow}_{n,d}(\P(V))$ (see \cite[Section 2.1]{Don01}).
This will be relevant later when applying GIT for vector spaces on the Chow variety.

\section{Finite dimensional arguments}\label{sec : fin dim arguments}

As in the previous section, we fix a complete discrete valuation ring $R$ with fraction field $K$ and residue field $\C$ and choose a valuation $\mathsf{v} : K^\times \to \R$.
For simplicity, we assume that the valuation is \textit{normalized}, meaning that $\mathsf{v}(K^*) = \Z$.
This may always be obtained by rescaling.
We are interested in the GIT setup of $G$ acting linearly on projective space $(\P(H),\O(1))$, where $G$ is the general linear group $\GL(V)$ or the special linear group $\SL(V)$, for finite-dimensional complex vector spaces $V$ and $H$.

The goal of this section is to provide a general construction of a one-parameter subgroup out of two $R$-points in $\P(H)(R)$ together with a fixed identification in $G(K)$, such that its GIT weight is the difference of GIT heights.
This may be viewed as a GIT version of the construction of a filtration due to Blum--Xu \cite{BX19} and Hattori \cite{Hat24}.
As an application (Corollary \ref{cor : GIT height minimization}), we prove \textit{GIT height minimization}, which serves as a finite-dimensional analogue of CM minimization \cite{Wan12,Oda12,WX12,Hat24}, strengthening work due to Wang \cite{Wan12} and Wang--Xu \cite{WX12}.
Lastly, we provide an application on separatedness of projective GIT quotients.

\subsection{GIT height for vectors}

Recall that, given a finite-dimensional vector space $E$ and a $\C$-algebra $A$, the space of $A$-points in the general linear group scheme $\GL(E)$, the affine variety $E$ and the projective variety $\P(E)$ are given by
\begin{align*}
    \GL(E)(A) &= \GL(E_A),
    \\
    E(A) &= E_A,
    \\
    \P(E)(A) &= \P(E_A),
\end{align*}
where $E_A = E \otimes_\C A$ denotes the base-change.
In the case $A = R$, one has natural evaluation map $E(R) \to E(\C)$, $\varphi \mapsto \varphi(0)$ induced by the morphism $R \to \C$ into the residue field.

Let $H$ be a finite-dimensional complex representation of $\GL(V)$ for an $N+1$-dimensional complex vector space $V$.

\subsubsection{The setup}

Let $\varphi$ and $\psi$ be $R$-points of $H$ and assume that the evaluations $\varphi(0)$ and $\psi(0)$ are non-zero.
This implies that $\varphi$ and $\psi$ define $R$-points $\P(H)$, which we denote by $[\varphi]$ and $[\psi]$.
Moreover, assume that there is a $K$-point $g$ in $\GL(H)$ such that $g.[\varphi] = [\psi]$ in $\P(H_K)$.
In particular, there exists a non-zero element $\gamma$ in $K$, such that
\begin{equation*}
    g.\varphi
    = \gamma \cdot \psi.
\end{equation*}

Given an inner product $\langle \cdot,\cdot \rangle$ on $H$, we have natural identifications
\begin{equation*}
\begin{aligned}
    H^0(\P(H),\O(1))
    &\simeq H,
    \\
    H^0(\Spec R, \varphi^* \O(1))
    &\simeq R \cdot \varphi,
\end{aligned}
\end{equation*}
where we wrote $\varphi : \Spec R \to \P(H)$ for the morphism, by slight abuse of notation.
Under these identifications, the restriction of global sections takes the form
\begin{align*}
    H^0(\P(H),\O(1)) &\to H^0(\Spec R, \varphi^*\O(1)),
    \\
    \quad \sigma &\mapsto \langle \varphi , \sigma \rangle
    \cdot \varphi,
\end{align*}
where $\langle \cdot,\cdot \rangle : \overline{\!H}_{\!R} \otimes H \to R$ denotes the base-change.
The isomorphism $\Phi_g : H^0(\Spec K, \varphi^*\O(1)) \to H^0(\Spec K,\psi^*\O(1))$ obtained by composing with $g$ satisfies
\begin{equation*}
    \Phi_g(f \cdot \varphi)
    = \gamma f \cdot \psi,
\end{equation*}
for all $f$ in $R$.

Therefore, the difference of GIT heights between $\varphi$ and $\psi$, with respect to the identification $g$, for $G = \GL(V)$ and $(X,L) = (\P(H),\O(1))$, is
\begin{equation}\label{eq : difference of GIT heights for vectors}
    h(\varphi) - h(\psi)
    = \mathsf{v} (\gamma).
\end{equation}
Note that this number only depends on the choice of points $[\varphi]$ and $[\psi]$ in projective space $\P(H)(R)$ and $g$.

\subsubsection{Construction of a one-parameter subgroup}\label{subsubsec : construc of 1-PS}

We now construct a one-parameter subgroup of $\GL(V)$ out of $g$ by defining a filtration $\mathcal{F}$ of $V$.
Given a finite-dimensional vector space $E$ and vector $\xi$ in $E_K$, we define
\begin{equation*}
    \mathsf{v}(\xi) = \min_{0 \leq i \leq \dim E} \mathsf{v}(\xi^i),
\end{equation*}
where $\xi = \sum_{i = 1}^{\dim E} \xi^i e_i$ for elements $\xi^i$ in $K$, given any basis $(e_i)$ of $E$.
In other words, given a uniformizer $t$ in $R$, the integer $\mathsf{v}(\xi)$ is the unique largest integer such that the vector $t^{-\mathsf{v}(\xi)}\xi$ is contained in $E_R$.

Define a decreasing sequence of submodules and linear subspaces
\begin{equation}\label{eq : def GIT filtrations}
\begin{aligned}
    \mathcal{F}^\lambda V_R
    &= \big\{ v \in V_R \;|\; \mathsf{v}(g.v) \geq \lambda  \big\},
    \\
    \mathcal{F}^\lambda V
    &= \mathrm{im} \big( \mathcal{F}^\lambda V_R \to V,
    \, v \mapsto v(0) \big),
\end{aligned}
\end{equation}
for integers $\lambda$.
Note that $\mathcal{F}^\lambda V_R = V_R$ for sufficiently small $\lambda$ and $\mathcal{F}^\lambda V_R = 0$ for sufficiently large $\lambda$.
We obtain a (bounded) filtration $\mathcal{F}^\bullet V$ of $V$ (see \cite[Section 1.2]{BC11}), which gives rise to a $\Z$-grading $V = \bigoplus_{\lambda \in \Z} V_\lambda$, for
\begin{equation*}
    V_\lambda = \big( \mathcal{F}^{\lambda + 1}V \big)^\perp \cap \mathcal{F}^\lambda V,
\end{equation*}
where $(-)^\perp$ denotes taking the orthogonal complement in $V$.

Let $\rho : \C^* \to \GL(V)$ be the one-parameter subgroup associated to the grading.
This is given by $\rho(z)(v_\lambda) = z^\lambda v_\lambda$ for $z$ in $\C^*$ and $v_\lambda$ in $V_\lambda$.

The following lemma will be useful when relating $g$ with $\rho$:
\begin{lemma}\label{lem : order of g vs order of rho}
    For all $\sigma$ in $H_R$ with $\sigma(0) \neq 0$, we have
    \begin{equation*}
        \mathsf{v}(g.\sigma)
        = \mathrm{ord}_0(\rho(z).\sigma(0)),
    \end{equation*}
    where the right-hand side denotes the order of the holomorphic map $\C^* \to H$, sending $z \mapsto \rho(z).\sigma(0)$.
\end{lemma}
\begin{proof}
    By choosing a uniformizer $t$ in $R$, we may assume that $R = \C [\![ t ]\!]$ and $K = \C (\!( t )\!)$.
    In that case, we may restrict $\rho$ to a $K$-point $\rho_K$.
    The product $\theta = g^{-1} \rho|_K$ is a $K$-point of $\GL(V)$.
    By construction of $\rho$, we have
    \begin{equation*}
        \mathsf{v}(g.v)
        = \mathrm{ord}_0(\rho(z).v(0)),
    \end{equation*}
    for all $v$ in $V_R$ with $v(0) \neq 0$, where the order at $0$ is defined same way as above.
    This implies that $\theta$ extends to an $R$-point of $\GL(V)$.
    Using this, we have
    \begin{align*}
        \mathrm{ord}_0(\rho(z).\sigma(0))
        &= \mathsf{v}(\rho_K.\sigma),
        \\
        &= \mathsf{v}((g  \theta).\sigma),
        \\
        &= \mathsf{v}(g.\sigma),
    \end{align*}
    as desired.
\end{proof}

\subsubsection{Weight-height results and applications}

We now compute the GIT weight of the one-parameter subgroup $\rho$ constructed above and prove that it is indeed the difference of GIT heights.
Let $A = \frac{d}{dt}|_{t = 0} \rho(e^t)$ be the infinitesimal generator of $\rho$, regarded as an element of $\mathfrak{gl}(V) = \End(V)$.
\begin{theorem}\label{thm : git weight and height for vectors}
    The GIT weight of $\rho$ at $\varphi(0)$ with respect to the $\mathrm{GL}(V)$-action on $(\P(H),\O(1))$ and the total weight of the $\rho(\C^*)$-action on $V$ satisfy
    \begin{align*}
        \nu(\rho,\varphi(0))
        &= h(\psi) - h(\varphi),
        \\
        \tr(A)
        &= \mathsf{v}(\det g).
    \end{align*}
    In particular, if $g$ is a $K$-point of $\SL(V)$, then $\rho(\C^*)$ is contained in $\SL(V)$.
\end{theorem}

\begin{proof}
    The one-parameter subgroup $\rho : \C^* \to \GL(V)$ and the representation $\GL(V) \to \GL(H)$ give rise to a $\C^*$-action on $H$, which in turn corresponds to a $\Z$-grading $H = \bigoplus_{\lambda \in \Z} H_\lambda$.
    Write $\varphi(0) = \sum_{\lambda \in \Z} \varphi(0)_\lambda$ for the decomposition of $\varphi(0)$ with respect to this $\Z$-grading.
    By definition of the GIT weight and Lemma \ref{lem : order of g vs order of rho}, we have
    \begin{align*}
        \nu(\rho,\varphi(0))
        &= -\max \{ \lambda \in \Z \;|\; \varphi(0)_\lambda \neq 0 \},
        \\
        &= - \mathrm{ord}_0(\rho(z).\varphi(0)),
        \\
        &= - \mathsf{v}(g.\varphi),
        \\
        &= - \mathsf{v}(\gamma),
        \\
        &= h(\psi) - h(\varphi).
    \end{align*}
    
    For the second equation, note that $\tr(A)$ is equal to the order at $0$ of the regular function $\det \rho : \C^* \to \C^*$.
    Under the identification $\C^* \simeq \GL(\det V)$, the determinant $\det \rho$ is the one-parameter subgroup constructed out of the $K$-point $\deg g$ of $\GL(\det V)$.
    Applying Lemma \ref{lem : order of g vs order of rho} to the case $H = \det V$, we obtain
    \begin{equation*}
        \mathrm{ord}_0 ( \det \rho )
        = \mathsf{v}(\det g),
    \end{equation*}
    as desired.
    The last statement simply follows from the fact that $\rho(\C^*)$ is contained in $\SL(V)$ if and only if $\tr(A) = 0$.
\end{proof}

Let $\underline{\rho} : \C^* \to \SL(V)$ be the one-parameter subgroup associated to the infinitesimal generator
\begin{equation*}
    \underline{A} = A - \frac{\tr (A)}{N+1}\mathrm{id}_V,
\end{equation*}
the trace-free part of $A$, which lies in $\mathfrak{sl}(V)$.
Note that $\rho(\C^*)$ is contained in $\SL(V)$ if and only if $\rho = \underline{\rho}$.
\begin{corollary}\label{cor : SL git weight for vectors}
    The GIT weight of $\underline{\rho}$ at $\varphi(0)$ with respect to the $\SL(V)$-action on $(\P(H),\O(1))$, is
    \begin{equation*}
        \nu_\mathrm{SL}(\underline{\rho},\varphi(0))
        = \big( h(\psi) - h(\varphi) \big)
        + \frac{\mathsf{v}(\det g)}{N+1}.
    \end{equation*}
    In particular, if $g$ is a $K$-point of $\SL(V)$, then
    \begin{equation*}
        \nu_\mathrm{SL}(\rho,\varphi(0))
        = h(\psi) - h(\varphi).
    \end{equation*}
\end{corollary}
\begin{proof}
    Choose an inner product $\langle \cdot, \cdot \rangle$ of $H$ such that the $\mathrm{U}(V)$-action on $H$ is unitary and let $\mu : \P(H) \to \mathfrak{u}(V)^\vee$ be the moment map for the $\mathrm{U}(V)$-action on $(\P(H),\omega_\mathrm{FS})$, defined the same way as in Equation \eqref{eq : moment map mu_H}.
    
    Using Equation \eqref{eq : GIT weight = moment map}, we obtain
    \begin{align*}
        \nu_\mathrm{SL}(\underline{\rho},\sigma)
        &= - \mu(\sigma_0)(i\underline{A}),
        \\
        &= - \Big( \mu(\sigma_0)(iA) 
        - \frac{i\tr A}{N+1} \frac{1}{i} \Big),
        \\
        &= \nu(\rho,\sigma)
        + \frac{\tr A}{N+1},
    \end{align*}
    for any $\sigma$ in $H$, where we wrote $\sigma_0 = \lim_{t \to 0} \rho(t).\sigma$.
    Applying Theorem \ref{thm : git weight and height for vectors} for $\sigma = \varphi(0)$ yields the result.
\end{proof}

As an application of these weight-height results, we prove GIT height minimization, which gives rise to a numerical proof for separatedness of the projective GIT quotient for the $\GL(V)$ or $\SL(V)$-action on $\P(H)$.
\begin{corollary}\label{cor : GIT height minimization}
    Let $G = \GL(V)$ or $\SL(V)$ and assume that $g$ is a $K$-point of $G$.
    If the point $[\varphi(0)]$ in $\P(H)$ is $G$-semistable, then
    \begin{equation*}
        h(\varphi) \leq h(\psi).
    \end{equation*}
    If $[\varphi(0)]$ is $G$-polystable, then, equality implies that $[\varphi(0)]$ is contained in the closure of the $G$-orbit of $[\psi(0)]$.
    If $[\varphi(0)]$ is $G$-stable, then, equality implies that $[\varphi(0)] = [\psi(0)]$.
\end{corollary}
This strengthens work due to Wang \cite[Theorem 4]{Wan12} and Wang--Xu \cite[Theorem 4]{WX12}, by allowing families over non-proper bases and proving strict minimization of polystable fillings.
The result is based on the construction of a one-parameter subgroup, which is different to the techniques adopted by Wang and Wang--Xu.
\begin{proof}
    This is an application of Theorem \ref{thm : git weight and height for vectors} and Corollary \ref{cor : SL git weight for vectors}.
    If $[\varphi(0)]$ is semistable, then $\nu_G(\rho,\varphi(0)) \geq 0$, which shows the inequality.

    For the equality statements, we require some theory about the specializations of $[\varphi(0)]$ and $[\psi(0)]$ with respect to $\rho$.
    Similar to Equation \ref{eq : def GIT filtrations}, define filtrations
    \begin{equation*}
    \begin{aligned}
        \mathcal{F}^\lambda H_R
        &= \big\{ \sigma \in H_R \;\big|\; \mathsf{v}(g.\sigma) \geq \lambda  \big\},
        \\
        \mathcal{F}^\lambda H
        &= \mathrm{im} \big( \mathcal{F}^\lambda H_R \to V,
        \, \sigma \mapsto \sigma(0) \big).
    \end{aligned}
    \end{equation*}
    Note that in the proof of Theorem \ref{thm : git weight and height for vectors}, we proved that $\mathcal{F}^\bullet H$ is the filtration associated to the one-parameter subgroup $\rho : \C^* \to \GL(V) \to \GL(H)$.
    This means that $\mathcal{F}^\lambda H = \bigoplus_{\mu \geq \lambda} H_\mu$, and
    \begin{equation*}
        \nu_G(\rho,\varphi(0))
        = -\max \big\{ \lambda \in \Z \;\big|\; \varphi \in \mathcal{F}^\lambda H_R \big\}.
    \end{equation*}

    Let $\mathcal{F}'$ be the filtrations of $V_R$, $V$, $H_R$ and $H$, defined the same way, but using $g^{-1}$ instead of $g$.
    The graded pieces $\mathrm{gr}_\mathcal{F}^\lambda  E_R = \mathcal{F}^\lambda E_R / \mathcal{F}^{\lambda + 1} E_R$ for $E = V$ and $H$ admit isomorphisms
    \begin{equation}\label{eq : isomorphism of graded components}
        \mathrm{gr}_\mathcal{F}^\lambda  E_R
        \to \mathrm{gr}^{-\lambda}_{\mathcal{F}'} E_R,
        \quad [\xi] \mapsto [t^{-\lambda} g.\xi].
    \end{equation}
    This means that the $\Z$-grading $V = \bigoplus_{\lambda \in \Z} V'_{- \lambda}$, constructed out of $\mathcal{F}'$, satisfies $V_\lambda = V'_{-\lambda}$.
    In particular, the one-parameter subgroup of $\GL(V)$ constructed out $\mathcal{F}'$ coincides with the inverse $\rho^{-1}$, and
    \begin{equation}\label{eq : nu = -nu'}
    \begin{aligned}
        \nu_G(\rho,\varphi(0))
        &= h(\psi) - h(\varphi),
        \\
        &= - \big( h(\varphi) - h(\psi) \big),
        \\
        &= -\nu_G(\rho^{-1},\psi(0)).
    \end{aligned}
    \end{equation}
    Moreover, writing $\nu = \nu_G(\rho,\varphi(0))$, the coset $[\varphi]$ in $\mathrm{gr}^{-\nu}_\mathcal{F} H_R$ is non-zero, and maps to $[\psi]$ under Equation \eqref{eq : isomorphism of graded components}, also non-zero in $\mathrm{gr}^\nu_{\mathcal{F}'} V$.
    
    If $\nu = 0$, we obtain a commutative diagram
    \begin{equation*}
    \begin{tikzcd}
        \P(\mathrm{gr}^0_\mathcal{F} H_R) \ar[rr,"\sim"] \ar[d]
        && \P(\mathrm{gr}^0_{\mathcal{F}'} H_R) \ar[d]
        \\
        \P(\mathrm{gr}^0_\mathcal{F} H) \ar[dr, "\sim", swap]
        && \P(\mathrm{gr}^0_{\mathcal{F}'} H) \ar[dl, "\sim"]
        \\
        & \P(H_0),
    \end{tikzcd}
    \end{equation*}
    This implies that $[\varphi(0)_0] = [\psi(0)_0]$ in $\P(H)$, and hence
    \begin{equation}\label{eq : polystable degeneration}
        \lim_{z \to 0} \rho(z).[\varphi(0)]
        = \lim_{z \to 0} \rho(z)^{-1}.[\psi(0)].
    \end{equation}
    
    Now, if $[\varphi(0)]$ is polystable and $\nu = 0$, then $\rho(\C^*)$ is contained in the stabilizer of $[\varphi(0)]$ in $G$.
    Then, Equation \eqref{eq : polystable degeneration} implies that $[\varphi(0)]$ is contained in the closure of the $\rho(\C^*)$-orbit of $[\psi(0)]$.
    
    If $[\varphi(0)]$ is stable and $\nu = 0$, then $\rho$ is trivial.
    Then, Equation \eqref{eq : polystable degeneration} implies $[\varphi(0)] = [\psi(0)]$.
\end{proof}

\begin{remark}
    Equation \eqref{eq : polystable degeneration} is true whenever $\nu = 0$.
    For example, this is the case when $[\varphi(0)]$ and $[\psi(0)]$ are semistable, because of Equation \eqref{eq : nu = -nu'}. 
    The equality of specializations may be viewed as a GIT-analogue of the uniqueness of polystable degenerations of K-semistable Fano varieties \cite{BX19} (see also \cite[Section 8.2.2]{Xu25}).
\end{remark}

As an application of Corollary \ref{cor : GIT height minimization}, using the valuative criterion for separatedness of schemes, we obtain a numerical proof for separatedness of projective GIT quotients of polarized $G$-schemes for $G = \GL(V)$ or $\SL(V)$.
\begin{corollary}
    Let $Z$ be a $G$-invariant subschme of $\P(V)$. Then, the GIT quotient
    \begin{equation*}
        Z \!\sslash\! G
        = \Proj \Big( \bigoplus_{k \geq 0} H^0(Z,\O_Z(k))^G \Big),
    \end{equation*}
    parameterizing polystable orbits \cite[Corollary 5.28]{Hos15} is separated.
\end{corollary}

More generally, as any polarized $G$-variety may be embedded equivariantly into projective space, we also obtain separatedness of general projective GIT quotients with respect to an ample linearization.

We also obtain a version the moment-weight inequality from Equation \eqref{eq : moment-weight ineq} for $R$-points in $\P(V)$ and GIT height:
\begin{equation*}
    \inf_{g' \in G} \left\|\mu(g'.\varphi(0))\right\|
    \geq - \sup_{\psi} \frac{h(\psi) - h(\varphi)}{d(\psi,\varphi)},
\end{equation*}
where the supremum is taken over the space of $R$-points $\psi$ in $\P(V)$ endowed with an identification $g$ in $G(K)$, and we wrote $d(\psi,\varphi) = \| A \|$.

\subsection{Applications to Chow points}\label{subsec : applic to Chow pts}

Our main application of the theory in the previous section is for models and their associated Chow points.

Let $\X$ be a flat integral subscheme of $\P(V) \times \Spec R$, and denote by $n$ and $d$ the dimension and degree of the central fiber $\mathcal{X}_0$ in $\P(V)$.
Writing $r = N-n$, we are interested in the case
\begin{equation*}
    H = H^0 ( \mathrm{Gr}(V,r),\O(d)),
\end{equation*}
as in Section \ref{subsec : bergman kernel}.
The Chow $R$-point of $\X$ is an $R$-point $R_\X$ of $\P(H)$.

Let $\mathcal{X}'$ be another flat integral subscheme of $\P(V) \times \Spec R$ and assume that there is a $K$-point $g$ in $\GL(V)$ such that the  induced automorphism $g : V_K \to V_K$ restricts to an isomorphism $(\X_K,\O_{\X_K}(1)) \to (\X'_K,\O_{\X'_K}(1))$.
This is equivalent to giving two models as established in Lemma \ref{lem : models and K-points of GL}.

A basic property of the Chow points is equivariance, meaning that the restrictions of the Chow $R$-points $R_\mathcal{X}$ and $R_{\mathcal{X}'}$ to $\Spec K$ satisfy
\begin{equation*}
    g.R_{\X_K} = R_{\X'_K},
\end{equation*}
in $\P(H_K)$.
In particular, we may apply the theory of the previous section to the case $[\varphi] = R_\X$ and $[\psi] = R_{\X'}$ with identification $g$.

Let $\rho : \C^* \to \GL(V)$ be the one-parameter subgroup from Section \ref{subsubsec : construc of 1-PS}, constructed out of $g$, and let $A$ be its infinitesimal generator in $\End(V)$.
\begin{theorem}\label{thm : GIT weight of rho for Chow points}
    The GIT weight of $\rho$ at $R_{\X_0}$ with respect to the natural $\GL(V)$-action on $(\P(H),\O(1))$ satisfies
    \begin{equation*}
        \nu(\rho,R_{\X_0})
        = h(R_{\X'}) - h(R_{\X}).
    \end{equation*}
    Moreover, we have
    \begin{equation*}
        \tr (A)
        = \deg \big( \det \pi_* \O_\X(1) - \det \pi_* \O_{\X'}(1) \big),
    \end{equation*}
    where $\X \xrightarrow{\pi} \Spec R$ and $\X' \xrightarrow{\pi'} \Spec R$ denote the structure morphisms.
\end{theorem}
The degree in the second equality is taken with respect to the isomorphism
\begin{equation*}
    \Psi_g : H^0(\Spec K, \det \pi_* \O_\X(1))
    \to H^0(\Spec K, \det \pi'_* \O_{\X'}(1)),
\end{equation*}
given by the isomorphism $(\X_K,\O_{\X_K}(1)) \to (\X'_K,\O_{\X'_K}(1))$.
\begin{proof}
    The first equality is simply Theorem \eqref{thm : git weight and height for vectors}.
    For the second equality, note that we have canonical isomorphisms
    \begin{align*}
        H^0(\Spec R, \det \pi_* \O_\X(1) ) \simeq \det V_R,
        \\
        H^0(\Spec K, \det \pi_* \O_\X(1) ) \simeq \det V_K.
    \end{align*}
    Under these identifications, the isomorphism $\Psi_g$ is simply multiplication by $\det g$, viewed as a non-zero element of $K$.
    In particular, we have
    \begin{equation*}
        \deg \big( \det \pi_* \O_\X(1) - \det \pi'_* \O_{\X'}(1) \big)
        = \mathsf{v}(\det g),
    \end{equation*}
    so the result follows by Theorem \ref{thm : git weight and height for vectors}, again.
\end{proof}
\begin{remark}
    Note that the order of the models $\X$ and $\X'$ between the two equations in Theorem \ref{thm : GIT weight of rho for Chow points} are switched.
    This is intentional and will be relevant in the next corollary.
\end{remark}

The following line bundle was considered by Wang \cite[Section 3.1]{Wan12}.
\begin{definition}\label{def : chow line bundle}
    The \textit{Chow line bundle} associated to a flat integral subscheme $\X$ in $\P(V) \times \Spec R$ is defined to be the $\Q$-line bundle
    \begin{equation*}
        \lambda_\mathrm{Chow}(\X)
        = R^*_\X \O(1) - \frac{\det \pi_*\O_{\X}(1)}{\dim V}.
    \end{equation*}
\end{definition}

As before, the $K$-point $g$ in $\GL(V)$ induces an isomorphism of $\Q$-line bundles $\lambda_\mathrm{Chow}(\X)|_{\Spec K} \simeq \lambda_\mathrm{Chow}(\X')|_{\Spec K}$, so the degree of the difference of Chow line bundles with respect to $g$ is defined.
Combining Corollary \ref{cor : SL git weight for vectors} with Theorem \ref{thm : GIT weight of rho for Chow points}, and using additivity of the degree, we obtain:

\begin{corollary}\label{cor : SL weight of Chow point}
    The GIT weight of $\underline{\rho}$ at $R_{\X_0}$ with respect to the natural $\SL(V)$-action on $(\P(H),\O(1))$ satisfies
    \begin{equation*}
        \nu_\mathrm{SL}(\underline{\rho},R_{\X_0})
        = \deg \big( \lambda_\mathrm{Chow}(\X')
        - \lambda_\mathrm{Chow}(\X) \big).
    \end{equation*}
\end{corollary}

\section{Main results}

We now come to the proof of Theorem \ref{thm : main theorem}, which will make use of the theory established in the previous sections.
Let $(X,L)$ be a polarized variety over $K$, write $V_k = H^0(\X_0,k\L_0)$ and let
\begin{equation*}
    \dim V_k
    = a_0 k^n +
    a_1 k^{n-1}
    + O(k^{n-2})
\end{equation*}
as $k \to +\infty$.
Let $(\X,\L)$ and $(\X',\L')$ be non-isomorphic models of $(X,L)$.
As in Section \ref{subsec : models and cm degree}, choose trivializations of vector bundles $\pi_*(k\L)^\vee \cong V_k \times \Spec R \cong \pi_*'(k\L')^\vee$, which, for $k$ very large, induce embeddings
\begin{align*}
    \iota_k : \X &\hookrightarrow \P(V_k) \times \Spec R,
    \\
    \iota'_k : \X' &\hookrightarrow \P(V_k) \times \Spec R.
\end{align*}

The special fibers of the images $\X_k = \iota_k(\X)$ and $\X_k' = \iota'_k (\X')$ in $\P(V_k)$ have dimension $n$ and degree $n!a_0 k^n$.
Moreover, the trivializations above and the identification $\pi_*(k\L)|_{\Spec K} \simeq \pi'_*(k\L')|_{\Spec K}$ give rise to a $K$-point $g_k$ in $\mathrm{GL}(V_k)$ as in Lemma \ref{lem : models and K-points of GL}.
The $K$-point $g_k$ has the property that it restricts to an isomorphism
\begin{equation*}
    g_k : (\X_{k,K},\O_{\X_{k,K}}(1)) \simeq (\X_{k,K}',\O_{\X'_{k,K}}(1)).
\end{equation*}

Assume that $\X_0$ is smooth, choose a K\"ahler metric $\omega$ in $c_1(\L_0)$ and a hermitian metric $h$ on $\L_0$ with curvature $\frac{i}{2\pi}\omega$, inducing $L^2$-inner products $\langle \cdot,\cdot \rangle$ on $V_k$.
Recall that the Lie algebra $\mathfrak{su}(V_k)$ admits a natural $p$-norm given by $\|A\|_p^p = \tr(A^p)$, which induces a norm on its dual.
Let $\rho_k : \C^* \to \GL(V_k)$ be the one-parameter subgroup from Section \ref{subsec : applic to Chow pts}, constructed out of $g_k$, with infinitesimal generator $A_k$ in $\mathrm{End}(V_k)$.
By combining Corollary \ref{cor : SL weight of Chow point} with Equations \eqref{eq : moment map between chow var and chow point} and the H\"older inequality, we obtain the following lower bound for the Chow moment map functional for any H\"older conjugate pair $1 < p,q < +\infty$:
\begin{equation}\label{eq : lower bound of chow moment map}
    \left\|\mu_\mathrm{Chow}([\X_{k,0}])\right\|_q
    \geq - a_0\frac{\deg \big(\lambda_\mathrm{Chow}(\X_k') - \lambda_\mathrm{Chow}(\X_k)\big)}{\|\underline{A}_k\|_p} k^n.
\end{equation}
Here, the degree is taken with respect to the identification $g_k$.

In order to relate Equation \eqref{eq : lower bound of chow moment map} with the lower bound in Section \ref{subsec : strategy of the proof}, we require to express the difference of Chow heights with the difference of CM degrees and the norm of the infinitesimal generator with the distance between two models.
We do this in two sections, where we prove the required asymptotic expansions.

\subsection{The Chow degree}

Let $\lambda_\mathrm{Chow}(\X_k)$ be the Chow line bundle of the flat subscheme $\X_k$ in $\P(V_k) \times \Spec R$, as in Definition \ref{def : chow line bundle}, and let $\lambda_\mathrm{CM}(\X,\L)$ be the CM line bundle of $(\X,\L)$ as in Equation \eqref{eq : def CM line bundle}.

The following result is known due to Zhang \cite{Zha96} and Fine--Ross \cite{FR06}.
\begin{proposition}\label{prop : asymp of chow line bundle}
    There is a functorial isomorphism of $\Q$-line bundles
    \begin{equation*}
        \lambda_\mathrm{Chow}(\X_k)
        \simeq \lambda_\mathrm{CM}(\X,\L) + O(k^{-1}),
    \end{equation*}
    where $O(k^{-1})$ denotes a linear combination consisting of $\Q$-line bundles and factors of $k^m$ for $m \leq -1$.
    In particular, as $k \to +\infty$, we have
    \begin{equation*}
        \deg \big(\lambda_\mathrm{Chow}(\X_k') - \lambda_\mathrm{Chow}(\X_k)\big)
        = \big( \mathrm{CM}(\X',\L')
        - \mathrm{CM}(\X,\L) \big)
        + O(k^{-1}).
    \end{equation*}
\end{proposition}
Here, ``functorial'' means that the isomorphism commutes with the restricted isomorphisms on $\Spec K$, induced by $g_k$ on each side.
\begin{proof}
    Combining results of Zhang \cite[Theorem 1.4]{Zha96} and Boucksom--Eriksson \cite[Corollary A.23]{BE21}, we have a canonical functorial ismorphism
    \begin{equation*}
        R^*_{\X_k} \O(1)
        \simeq  \frac{k}{(n+1)!a_0} \lambda_0,
    \end{equation*}
    where $\lambda_0$ is the leading $\Q$-line bundle in the Knudsen--Mumford expansion for $(\X,\L)$ as in Equation \eqref{eq : Knudsen-Mumford expansion}.
    Using, moreover, that $\iota^*_k \O_{\X_k}(1) \simeq k\L$ functorially, we obtain
    \begin{align*}
        \lambda_\mathrm{Chow}(\X_k)
        &\simeq \frac{k\lambda_0}{(n+1)!a_0}
        - \frac{\det \pi_*(k\L)}{\dim V_k},
        \\
        &\simeq \frac{ \big( \frac{(n+1)a_0}{2}\lambda_1 + a_1 \lambda_0 \big)}{(n+1)!a_0^2}
        + O(k^{-1}),
        \\
        &\simeq \lambda_\mathrm{CM}(\X,\L) + O(k^{-1}),
    \end{align*}
    as desired.
\end{proof}

\subsection{The norm}\label{subsec : the norm}

We now wish to find an asymptotic expansion for the $p$-norms $\|\underline{A}_k\|_p$ as $k \to +\infty$.
For this, we have to explore a geometric interpretation of the filtration from Section \ref{subsubsec : construc of 1-PS} associated to $A_k$.

Let $\mathcal{F}_\mathrm{GIT}$ be the filtrations of $V_{k,R}$ and $V_k$, constructed out of $g_k$, as in Equation \eqref{eq : def GIT filtrations}.
There is a second way to obtain a filtration of $V_k$ given two models $(\X_k,\O_{\X_k}(1))$ and $(\X_k',\O_{\X_k'}(1))$, a construction due to Hattori \cite{Hat24} and Blum--Xu \cite{BX19}, as follows:
the embedding of $\X_k$ into $\P(V_k) \times \Spec R$, gives rise to a restriction map
\begin{equation}\label{eq : restriction map}
    V_{k,R} \to H^0(\X_k,\O_{\X_k}(1)),
    \quad v \mapsto s_v,
\end{equation}
via the natural isomorphism of $R$-modules $V_{k,R} = H^0(\P(V_k) \times \Spec R, \O(1))$.
Writing $(g_k)_* : H^0(\X_{k,K},\O(1)) \to H^0(\X'_{k,K},\O(1))$ for the push-forward of sections and choosing a uniformizer $t$ in $R$, we define
\begin{equation*}
\begin{aligned}
    \mathcal{F}_{\L,\L'}^\lambda V_{k,R}
    &= \big\{ v \in V_{k,R} \;\big|\; t^{-\lambda}(g_k)_* (s_v) \in H^0(\X'_k,\O_{\X'_k}(1)) \big\},
    \\
    \mathcal{F}_{\L,\L'}^\lambda V_k
    &= \mathrm{im} \big( \mathcal{F}^\lambda V_{k,R} \to V_k, v \mapsto v(0) \big).
\end{aligned}
\end{equation*}
The filtration $\mathcal{F}_{\L,\L'}$ is a (multiplicative, linearly bounded $\Z$-)filtration of the section ring $\bigoplus_{k \geq 0} V_k$ of $(\X_0,\L_0)$ in the sense of \cite{Nys12,Sze15}.

\begin{lemma}
    The filtrations $\mathcal{F}_\mathrm{GIT}$ and $\mathcal{F}_{\L,\L'}$ coincide.
\end{lemma}
\begin{proof}
    After choosing coordinates for $\P(V_k)$, the $R$-module $V_{k,R}$ is free of rank $N_k+1$ and generated by variables $x_0,\dots,x_{N_k}$ corresponding to hyperplane sections.
    In particular, since the embedding of $\X_k$ into $\P(V_k) \times \Spec R$ is given by a complete linear system, meaning that the restriction map in Equation \eqref{eq : restriction map} is surjective.
    Hence,
    \begin{equation*}
    \begin{aligned}
        \mathcal{F}_{\L,\L'}^\lambda V_{k,R}
        &= \big\{ v \in V_{k,R} \;\big|\; t^{-\lambda} (g_k)_*(s_v) \in H^0(\X'_k,\O_{\X'_k}(1)) \big\},
        \\
        &= \big\{ v \in V_{k,R} \;\big|\; t^{-\lambda} g_k.v \in V_{k,R} \big\},
        \\
        &= \big\{ v \in V_{k,R} \;\big|\; \mathsf{v}(g_k.v) \geq \lambda \big\},
        \\
        &= \mathcal{F}^\lambda_\mathrm{GIT} V_{k,R}.
    \end{aligned}
    \end{equation*}
    This implies $\mathcal{F}^\lambda_\mathrm{GIT} V_k = \mathcal{F}^\lambda_{\L,\L'} V_k$, as desired.
\end{proof}

We briefly recall the definition of the $L^p$-norm of a filtration on a section ring using Okounkov bodies and concave transforms \cite{Nys12,Sze15}:
given (any choice of) local coordinates of $\X_0$ and a non-vanishing global section of $\L_0$, one may construct a convex body $P$ in $\R^n$, known as the \textit{Okounkov body} of $(X,L)$.
Given a ($\Z$-)filtration $\F$ of the section ring $\bigoplus_{k \geq 0} V_k$, one may define a function $G_\F : P \to \R \cup \{ + \infty \}$, the \textit{concave transform} of $\F$.
While the precise definitions are less important for our applications, the crucial properties are that $a_0 = \mathrm{vol}(P)$ and $w_\mathcal{F}^p(k) = ( \int_P G_\mathcal{F}^p) k^{n+p} + o(k^{n+p})$ as $k \to +\infty$, where
\begin{equation*}
    w_\mathcal{F}^p(k)
    = \sum_{\lambda \in \Z} \lambda^p \dim (\mathcal{F}^\lambda V_k / \mathcal{F}^{\lambda + 1} V_k).
\end{equation*}

The \textit{$L^p$-norm} of $\F$ is given by the integral
\begin{equation*}
    \|\F\|^p_p
    = \frac{1}{\vol(P)}\int_P \big( G_\F - \overline{G}_\F \big)^p,
\end{equation*}
with respect to the Lebesgue measure, where $\overline{G}_\F = \int_P G_\F / \mathrm{vol}(P)$ denotes the average.
Note that in the case $p = 2$, we have
\begin{equation*}
    \|\F\|^2_2
    = \frac{c_0 a_0 - b_0^2}{a_0^2},
\end{equation*}
where we wrote $a_0$ is as before and $b_0$ and $c_0$ come from the expansions $w^1_\F(k) = b_0 k^{n+1} + o(k^{n+1})$ and $w^2_\F(k) = c_0 k^{n+2} + o(k^{n+2})$ as $k \to +\infty$.

\begin{proposition}\label{prop : asymp for norm of A_k}
    As $k \to +\infty$, we have
    \begin{equation*}
        \left\| \underline{A}_k \right\|_p
        = \sqrt[p]{a_0} \left\|\mathcal{F}_{\L,\L'} \right\|_p \, k^{\frac{n}{p}+1}  + o(k^{\frac{n}{p}+1}).
    \end{equation*}
\end{proposition}
\begin{proof}
    Since $\F_\mathrm{GIT} = \F_{\L,\L'}$, the dimension of the quotient $\mathcal{F}_{\L,\L'}^\lambda V_k / \mathcal{F}_{\L,\L'}^{\lambda + 1} V_k$ is equal to the multiplicity of $\lambda$ as an eigenvalue of $A_k$.
    In particular, we have $\mathrm{tr}(A_k^p) = w_{\mathcal{F}_{\L,\L'}}^p(k)$.
    We thus compute
    \begin{align*}
        \left\|\underline{A}_k\right\|^p_p
        &= \mathrm{tr} \left( \underline{A}_k^p \right),
        \\
        &= \sum_{j = 0}^p \binom{p}{j} 
        \tr (A_k^j) \left(
        - \frac{\tr A_k}{\dim V_k} \right)^{p-j},
        \\
        &= \sum_{j = 0}^p \binom{p}{j} \left( \int_P G_\F^j \right) \left(- \overline{G}_\F \right)^{p-j} k^{n+p}
        + o(k^{n+p}),
        \\
        &= a_0 \left\| \F_{\L,\L'} \right\|^p_p k^{n+p}
        + o(k^{n+p}).
    \end{align*}
    Taking roots proves the result.
\end{proof}

The rest of this section aims to prove that the $L^p$-norm of $\mathcal{F}_{\L,\L'}$ is equal to the normalized $d_p$-distance between the models as in Definition \ref{def : distance between models}.

Given a filtration $\F$ of $\bigoplus_{k \geq 0} V_k$, one may construct a non-Archimedean norm $\|\cdot\|_{\mathcal{F},k}$ on $V_k$, with respect to the trivial non-Archimedean absolute value $| \cdot |_\triv$ on $\C$, by declaring
\begin{equation*}
    \| s \|_{\F,k}
    = \exp \left( - \sup \left\{ \lambda \in \R \;|\; s \in \mathcal{F}^\lambda V_k \right\} \right),
\end{equation*}
for all sections $s$ in $V_k$.
The family of norms $\phi_\F = \{ \|\cdot\|_{\F,k}\}_k$ is graded bounded and the normalized $d_p$-distance between families of bounded graded norms on $\bigoplus_{k \geq 0} V_k$ is defined the same way as in the discretely valued case in Section \ref{subsubsec : dist between models}.

\begin{proposition}\label{prop : filtration norm = triv dist}
    Letting $\phi_\triv = \{ \| \cdot \|_{\mathrm{triv},k} \}$ be the family of trivial norms, we have
    \begin{equation*}
        \|\F\|_p = \hat{d}_p^\NA(\phi_\F,\phi_\mathrm{triv}).
    \end{equation*}
\end{proposition}
\begin{proof}
    The filtration $\F$ of $V_k$ and the choice of inner product on $V_k$ give rise to a one-parameter subgroup $\rho_k(\F) : \C^* \to \GL(V_k)$ with infinitesimal generator $A_k(\F)$ and trace-free part $\underline{A}_k(\F)$.
    The same computation as in the proof of Proposition \ref{prop : asymp for norm of A_k} yields
    \begin{equation*}
        \left\| \underline{A}_k(\F) \right\|^p_p
        = a_0 \|\F\|^p_p k^{n + p} + o(k^{n + p}).
    \end{equation*}
    
    Let $(e_i)$ be any orthonormal basis of $V_k$ that diagonalizes $\| \cdot \|_{\F,k}$, and write $\lambda_i = -\log \|e_i\|_{\F,k}$ and $\bar{\lambda} = (\dim V_k)^{-1}\sum_i \lambda_i$.
    By definition of $\| \cdot \|_{\F,k}$, the basis $(e_i)$ diagonalizes $A_k(\F)$ with eigenvalues $\lambda_i$.
    This means that $(e_i)$ diagonalizes $\underline{A}_k(\F)$ with eigenvalues $\lambda_i - \bar{\lambda}$.
    This implies that
    \begin{align*}
        \hat{d}_p( \| \cdot \|_{\F,k} , \| \cdot \|_{\mathrm{triv},k})^p
        &= \frac{1}{\dim V_k}
        \left\| \underline{A}_k(\F) \right\|^p_p,
        \\
        &= \|\F\|^p_p \, k^p + o(k^p),
    \end{align*}
    and the result follows.
\end{proof}

Now, let $\phi_{\L,\L'} = \{ \| \cdot \|_{\L,\L',k} \}$ be the family of non-Archimedean norms induced by the filtration $\F_{\L,\L'}$.

\begin{proposition}\label{prop : triv dist = discr dist}
    We have 
    \begin{equation*}
        \hat{d}_p^\NA (\phi_{\L,\L'},\phi_\mathrm{triv})
        = \hat{d}_p^\NA (\phi_\L, \phi_{\L'}).
    \end{equation*}
\end{proposition}
\begin{proof}
    Let $(e_i)$ be a basis of $V_k$ diagonalizing $\| \cdot \|_{\L,\L',k}$.
    By definition of $\F_{\L,\L'}$, there exist sections $(\bar{e}_i)$ in $H^0(\X,k\L)$ such that $\bar{e}_i|_{\X_0} = e_i$ and $t^{-\lambda_i}(g_k)_*(\bar{e}_i)$ is contained in $H^0(\X',k\L')$, where $\|e_i\|_{\L,\L',k} = e^{-\lambda_i}$.

    Writing $\|\bar{e}_i\|_{\L,k} = e^{-\mu_i}$ for real numbers $\mu_i$, we obtain
    \begin{equation*}
        \|\bar{e}_i\|_{\L',k} = e^{-\mu_i + \lambda_i}.
    \end{equation*}
    Note that $(\bar{e}_i)$ is necessarily $K$-linearly independent, and hence is a $K$-basis of $H^0(X,kL)$ diagonalizing both $\|\cdot\|_{\L,k}$ and $\|\cdot\|_{\L',k}$, satisfying
    \begin{equation*}
        \| e_i \|_{\L,\L',k} = \frac{\|\bar{e}_i\|_{\L,k}}{\|\bar{e}_i\|_{\L',k}}.
    \end{equation*}
    The result follows by definition of the non-Archimedean $d_p$-distances.
\end{proof}

\begin{remark}
    Alternatively, the norm $\| \cdot \|_{\L,\L,k}$ may directly be obtained from the norms $\| \cdot \|_{\L,k}$ and $\| \cdot \|_{\L,k}$, via the reduction map 
    \begin{equation*}
        H^0(\X,k\L) / \{ s \;|\; \|s\|_{\L',k} < 1 \}
        \to H^0(\X_0,k\L_0),
    \end{equation*}
    where the quotient is endowed with the quotient norm induced by $\| \cdot \|_{\L,k}$.
    This reduction map appears in \cite[Section 2.1.10]{BGR84}. 
\end{remark}

Combining the results in this section, we have proven that
\begin{equation}\label{eq : final norm expansion}
    \left\| \underline{A}_k \right\|_p
    = \sqrt[p]{a_0} \, \hat{d}_p \big( (\X',\L') , (\X,\L) \big) k^{\frac{n}{p} + 1}
    + o(k^{\frac{n}{p} + 1})
\end{equation}
as $k \to +\infty$.

\subsection{Proof of the main theorem}

Combining the results in the two sections above, we may now prove Theorem \ref{thm : main theorem}, which we briefly restate in more generality.
Let $1 < p,q < +\infty$ be a H\"older conjugate pair.

\begin{theorem}\label{thm : refined main theorem in Section 4}
    Let $(X,L)$ be a smooth polarized variety over $K$ and let $(\X,\L)$ be a smooth model of $(X,L)$. Then, we have
    \begin{equation*}
        \inf_{\omega \in \H_{\L_0}}
        \big\| S(\omega) - \widehat{S} \big\|_{ L^q(\X_0,\omega)}
        \geq
        - \sup_{(\X',\L')} \frac{\mathrm{CM}(\X',\L') - \mathrm{CM}(\X,\L)}{\hat{d}_p \big( (\X',\L') , (\X,\L)\big)},
    \end{equation*}
    where $\H_{\L_0}$ denotes the space of K\"ahler metrics in $c_1(\L_0)$, and the supremum is taken over the space of models $(\X',\L')$ of $(X,L)$ not isomorphic to $(\X,\L)$.
\end{theorem}

\begin{proof}
    Putting the expansions in Proposition \ref{prop : asymp of chow line bundle} and Equation \eqref{eq : final norm expansion} into Equation \eqref{eq : lower bound of chow moment map}, we obtain the lower bound
    \begin{equation*}
        \|\mu_\mathrm{Chow}([\X_{k,0}])\|_q
        \geq - \sqrt[q]{a_0} \frac{\mathrm{CM}(\X',\L') - \mathrm{CM}(\X,\L)}{\hat{d}_p \big( (\X',\L') , (\X,\L) \big)} k^{\frac{n}{q}-1}
        + o(k^{\frac{n}{q}-1}),
    \end{equation*}
    as $k \to +\infty$.
    Combining this with the upper bound from Theorem \ref{thm : donaldsons bergman kernel expansion}, we obtain
    \begin{equation*}
        \big\|S(\omega) - \widehat{S} \big\|_{L^q(\X_0,\omega)}
        \geq - \frac{\mathrm{CM}(\X',\L') - \mathrm{CM}(\X,\L)}{\hat{d}_p \big( (\X',\L'),  (\X,\L) \big)},
    \end{equation*}
    for any K\"ahler metric $\omega$ in $c_1(\L_0)$ and model $(\X',\L')$ of $(X,L)$ not isomorphic to $(\X,\L)$.
\end{proof}

\subsection{A non-Archimedean interpretation of the result}\label{subsec : nA interpretation of results}

As mentioned in the introduction, one may view the difference of CM degrees in the sense of non-Archimedean geometry.
We briefly discuss the main ideas without going into details of the general theory, which may be found in many works, such as \cite{BHJ17,BE21,Reb21}.
In this section, we will restrict ourselves to explaining the interpretation of the result.
For simplicity, we will assume that the polarized $K$-variety $(X,L)$ and any model $(\X,\L)$ of $(X,L)$ are smooth and are the restrictions of families over an affine curve.
The results apply without any such assumptions.

Given such a model $(\X,\L)$, we define the \textit{Mabuchi energy} of $(\X,\L)$, to be the $\Q$-line bundle
\begin{equation*}
    M^\mathrm{NA}(\X,\L)
    = \frac{1}{2V} \Big( \big\langle K^\mathrm{log}_{\X/ \Spec R},\L^n \big\rangle - \frac{n\mu}{n+1} \big\langle \L^{n+1} \big\rangle \Big),
\end{equation*}
using the \textit{log canonical bundle} $K^\mathrm{log}_{\X/ \Spec R} = K_{\X/\Spec R} - (\X_0 - \X_{0,\mathrm{red}})$, where the brackets denote taking the Deline pairing (see \cite[Section 2.4]{Reb21}).
By work of Phong--Ross--Sturm \cite{PRS08} and Boucksom--Eriksson \cite[Appendix A]{BE21}, we have a natural functiorial isomorphism
\begin{equation*}
    M^\mathrm{NA}(\X,\L)
    \simeq \lambda_\mathrm{CM}(\X,\L)
    - \frac{1}{2V} \big\langle (\X_0 - \X_{0,\mathrm{red}}),\L^n \big\rangle.
\end{equation*}

The difference of Mabuchi energies between two models is defined by taking the degree with respect to the identification over the general point.
This means that
\begin{align*}
    \deg \big( M^\mathrm{NA}(\X',\L')
    &- M^\mathrm{NA}(\X,\L) \big)
    = \big( \mathrm{CM}(\X',\L') - \mathrm{CM}(\X,\L) \big)
    \\
    &- \frac{1}{2V} \Big( \big\langle (\X_0 - \X_{0,\mathrm{red}}),\L^n \big\rangle - \big\langle (\X_0' - \X'_{0,\mathrm{red}}),\L'^n \big\rangle \Big).
\end{align*}
Note that the divisor $\X_0' - \X'_{0,\mathrm{red}}$ is always effective and $\L'$ is relatively ample.
In particular, if $\X_0$ is reduced, we obtain the inequality
\begin{equation*}
    \deg \big( M^\mathrm{NA}(\X',\L')
    - M^\mathrm{NA}(\X,\L) \big)
    \leq \mathrm{CM}(\X',\L') - \mathrm{CM}(\X,\L).
\end{equation*}

From $(X,L)$, one may construct the Berkovich analytification $(X^\mathrm{an}, L^\mathrm{an})$ consisting of semivaluations on schematic points of $X$, extending the non-Archimedean absolute value $|\cdot| = e^{-\mathsf{v}}$ (see, for example \cite[Section 1]{BJ18b} or \cite[Section 2.3]{Reb21}).

Since $X$ is smooth and defined over a discretely valued non-Archimedean field of characteristic $0$, the polarized variety $(X,L)$ satisfies a property known as \textit{continuity of envelopes} \cite{BJ18} (see also \cite[Section 2.3.3]{Reb21}).
In this case, there is a one-to-one correspondence between (regularizable-from-below) plurisubharmonic metrics $\phi$ on $X^\mathrm{an}$ and the space submultiplicative bounded graded families of norms $\{ \| \cdot \|_{\phi,k} \}$ on $H^0(X,kL)$, modulo asymptotic equivalence \cite{Reb21b}.

Writing $\H^\NA_L$ for the space of metrics on $L^\mathrm{an}$, corresponding to families of norms induced by models (see Section \ref{subsubsec : dist between models}), we may define the \textit{relative Mabuchi energy} as a real-valued functional $M^\NA( \;\cdot\;,\phi_\L) : \H^\NA_L \to \R$, given by the degree
\begin{equation*}
    M^\NA(\phi_{\L'},\phi_\L)
    = \deg \big( M^\NA(\X',\L')
    - M^\NA(\X,\L) \big),
\end{equation*}
for any metric $\phi_{\L'}$ in $\H^\NA_\L$.
In the case that $\X_0$ is reduced, Theorem \ref{thm : refined main theorem in Section 4} then may be phrased purely in terms of non-Archimedean metrics as
\begin{equation*}
    \inf_{\omega \in \H_{\L_0}}
    \big\| S(\omega) - \widehat{S} \big\|_{ L^q(\X_0,\omega)}
    \geq
    - \sup_{\phi \in \mathcal{H}^\NA_L \setminus \{ \phi_\L \}} \frac{M^\NA(\phi,\phi_\L)}{\hat{d}_p ( \phi , \phi_\L)}.
\end{equation*}

\printbibliography

\end{document}